\documentclass{amsart}
\usepackage{graphicx}
\usepackage{amscd}
\usepackage{color}
\usepackage{amsmath}
\DeclareMathOperator{\Ima}{Im}

\newtheorem{theorem}{Theorem}[section]
\newtheorem{lemma}[theorem]{Lemma}

\newtheorem{corollary}[theorem]{Corollary}

\newtheorem{question}[theorem]{Question}

\theoremstyle{definition}

\newtheorem{remark}[theorem]{Remark}

\begin{document}

\title[The maximum and minimum genus of a multibranched surface]{The maximum and minimum genus of a multibranched surface}

\author{Mario Eudave-Mu\~{n}oz}
\address{Instituto de Matem\'{a}ticas, Universidad Nacional Aut\'onoma de M\'exico, Circuito Exterior, Ciudad Universitaria, 04510 M\'exico D.F.,  Mexico}
\email{mario@matem.unam.mx}
\thanks{The first author is partially supported by grant PAPIIT-UNAM 116720}

\author{Makoto Ozawa}
\address{Department of Natural Sciences, Faculty of Arts and Sciences, Komazawa University, 1-23-1 Komazawa, Setagaya-ku, Tokyo, 154-8525, Japan}
\email{w3c@komazawa-u.ac.jp}
\thanks{The second author is partially supported by Grant-in-Aid for Scientific Research (C) (No. 17K05262) and (B) (No. 16H03928), The Ministry of Education, Culture, Sports, Science and Technology, Japan}



\begin{abstract}
In this paper, we give a lower bound for the maximum and minimum genus of a multibranched surface by the first Betti number and the minimum and maximum genus of the boundary of the neighborhood of it, respectively.
As its application, we show that the maximum and minimum genus of $G\times S^1$ is equal to twice of the maximum and minimum genus of $G$ for a graph $G$, respectively.
This provides an interplay between graph theory and 3-manifold theory.
\end{abstract}

\maketitle


\section{Introduction}

\subsection{Definition of multibranched surfaces}

Let $\Bbb{R}^2_+$ be the closed upper half-plane $\{(x_1,x_2)\in \Bbb{R}^2 \mid x_2\ge 0\}$.
The {\em multibranched Euclidean plane}, denoted by $S_i$ $(i\ge 1)$, is the quotient space obtained from $i$ copies of $\Bbb{R}^2_+$ by identifying with their boundaries $\partial \Bbb{R}^2_+=\{(x_1,x_2)\in\Bbb{R}^2\mid x_2=0\}$ via the identity map.
See Figure \ref{model} for the multibranched Euclidean plane $S_5$.

\begin{figure}[htbp]
	\begin{center}
	\includegraphics[trim=0mm 0mm 0mm 0mm, width=.4\linewidth]{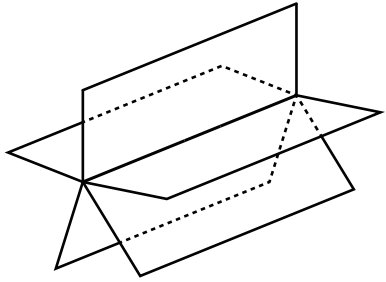}
	\end{center}
	\caption{The multibranched Euclidean plane $S_5$}
	\label{model}
\end{figure}


A second countable Hausdorff space $X$ is called a {\em multibranched surface} if $X$ contains a disjoint union of simple closed curves $l_1,\ldots, l_n$ satisfying the following:
\begin{itemize}
\item For each point $x\in l_1\cup \cdots \cup l_n$, there exist an open neighborhood $U$ of $x$ and a positive integer $i$ such that $U$ is homeomorphic to $S_i$.
\item For each point $x\in X-(l_1\cup\cdots\cup l_n)$, there exists an open neighborhood $U$ of $x$ such that $U$ is homeomorphic to $\Bbb{R}^2$.
\end{itemize}

\subsection{A construction of a multibranched surface by a covering map}

To construct a multibranched surface, 
we prepare a closed $1$-dimensional manifold $L$, a compact $2$-dimensional manifold $E$ and a continuous map $\phi: \partial E \to L$ satisfying the following conditions.
\begin{enumerate}
\item For every connected component $e$ of $E$, $\partial e \not= \emptyset$.
\item For every connected component $c$ of $\partial E$, the restriction $ \phi|_c : c \to \phi(c) $ is a covering map.
\end{enumerate}

The quotient space $X=L \cup_\phi E$ is called the {\it multibranched surface} obtained from the triple $(L, E; \phi)$. We note that $L$ and $E$ are not necessarily connected.

A connected component of $L$ (resp.~$E$, $\partial E$) is said to be a {\it branch} (resp.~{\it sector}, {\it prebranch}) of $X$. We note that every branch of $X$ is homeomorphic to the $1$-sphere $S^1$. The set consisting of all branches (resp.~sectors) is denoted by $\mathcal{L}(X)$ (resp.~$\mathcal{E}(X)$).

For a prebranch $c$ of a multibranched surface $X$, the covering degree of the covering map $\phi|_c:c\to \phi(c)$ is called the {\it degree} of $c$, denoted by $d(c)$. We note that $d(c)$ is a positive integer.
We give an orientation for each branch and each prebranch $c$ of $X$ (In the case that every sector $s$ is orientable, the orientation of $c$ is induced by that of $s$). 
The {\em oriented degree} of a prebranch $c$ of $X$ is defined as follows: if the covering map $\phi|_c:c\to \phi(c)$ is orientation preserving, the {\em oriented degree} $od(c)$ of $c$ is defined by $od(c)=d(c)$ and if it is orientation reversing, the oriented degree is defined by $od(c)=-d(c)$.

A prebranch $c$ of $X$ is said to be {\it attached} to a branch $l$ if $\phi(c)=l$.
We denote by $\mathcal{A}(l)$ the set consisting of all prebranches which are attached to a branch $l$ and the number of elements of $\mathcal{A}(l)$ is called the {\it index} of $l$, denoted by $i(l)$.

\subsection{Regular multibranched surfaces}

A multibranched surface $X$ is {\it regular} if for every branch $l$ and every prebranch $c$ and $c'$ of $X$ which are attached to $l$, $d(c)=d(c')$.

Let $X$ be a regular multibranched surface. Since each pair of prebranches $c, c'$ of $X$ which are attached to a branch $l$ has same degree, then we define the {\it degree} of a branch $l$ as $d(l)=d(c)=d(c')$.

It was shown in \cite{MO} (cf. \cite[Corollary 2.4]{RBS}) that a multibranched surface is embeddable in some closed orientable $3$-dimensional manifold if and only if the multibranched surface is regular.

\subsection{Circular permutation system and slope system}

In this paper, the cardinality of a set $S$ is denoted by $\# S$. A {\it permutation} of $S$ is a bijection from the additive group $\mathbb{Z} /n \mathbb{Z}$ into $S$. Two permutations $\sigma$ and $\sigma'$ of $S$ are {\it equivalent} if there is an element $k \in \mathbb{Z} /n \mathbb{Z}$ such that $\sigma'(x) = \sigma (x+k)$ ($x \in \mathbb{Z} /n \mathbb{Z}$). An equivalent class of a permutation of $S$ is a {\it circular permutation}.

For a regular multibranched surface $X$, we define the ``circular permutation system'' and ``slope system'' of $X$ as follows.
A circular permutation of $\mathcal{A}(l)$ is called a {\it circular permutation} on a branch $l$.
A collection $\mathcal{P}=\{ \mathcal{P}_l \}_{l \in \mathcal{L}(X)}$ is called a {\it circular permutation system} of $X$ if $\mathcal{P}_l$ is a circular permutation on $l$.
For a branch $l$, a rational number $p/q$ with $q=d(l)$ is called a {\it slope} of $l$.
A collection $\{ \mathcal{S}_l \}_{l \in \mathcal{L}(X)}$ is called a {\it slope system} of $X$ if $\mathcal{S}_l$ is a slope of $l$.

\subsection{Neighborhoods of a multibranched surface}

We define a compact $3$-dimensional manifold with boundary, called a ``neighborhood'' of a regular multibranched surface $X$.
This is uniquely determined up to homeomorphism by a pair of a circular permutation system $ \mathcal{P} = \{ \mathcal{P}_l \}_{ l \in \mathcal{L}(X) } $ and a slope system $\mathcal{S}=\{ \mathcal{S}_l \}_{ l \in \mathcal{L}(X) }$ of $X$.

Let $X=L \cup_\phi E$ be a regular multibranched surface and let $\mathcal{P}=\{ \mathcal{P}_l \}_{l \in \mathcal{L}(X)}$ and $\mathcal{S}=\{ \mathcal{S}_l \}_{ l \in \mathcal{L}(X) }$ be a permutation system and a slope system of $X$ respectively. We will construct the $3$-dimensional manifold by the following procedure. First, for each branch $l$ of $X$ and each sector $e$ of $X$ we take a solid torus $l \times D^2$, where $D^2$ is a disk and take a product $e \times [-1, 1]$. If $e$ is non orientable, we take a twisted $I$-bundle $e \tilde{\times} [-1, 1]$ over $e$. We give orientations for these $3$-dimensional manifolds. {Next, we glue them together depending on the permutation system $\mathcal{P}$ and the slope system $\mathcal{S}$, where we assign the slope $\mathcal{S}_l$ of $l$ to the isotopy class of a loop $k_l$ in $\partial ( l \times D^2)$,
by an orientation reversing continuous map $\Phi: \partial E \times [-1, 1] \to \partial (L \times D^2)$ satisfying that for every branch $l$ and every prebranch $c$ with $\phi(c)=l$, the restriction $\Phi |_{c \times [-1, 1]} : c \times [-1, 1] \to N \left(k_l; \partial \left( l \times D^2 \right) \right)$ is a homeomorphism.}
Then, we uniquely obtain a compact and orientable $3$-dimensional manifold with boundary, denoted by $N(X; \mathcal{P}, \mathcal{S})$. The $3$-dimensional manifold $N(X; \mathcal{P}, \mathcal{S})$ is called the {\it neighborhood} of $X$ with respect to $\mathcal{P}$ and $\mathcal{S}$. The set consisting of all neighborhoods of $X$ is denoted by $\mathcal{N}(X)$.

\subsection{The maximum and minimum genus of a multibranched surface}

A {\it compression body} $V$ is an orientable compact 3-manifold obtained from $F \times [0, 1]$ by attaching 2-handles along pairwise disjoint loops on $ F \times \{1\}$ and attaching 3-handles along some resulting spheres where $F$ is an orientable closed surface. The subspaces $F \times \{0\}$ and $\partial V - (F \times \{0\})$ of $V$ are denoted by $\partial_+ V$ and $\partial_- V$ respectively. We note that $V$ is a handlebody if $\partial_- V = \emptyset $. For any orientable compact 3-manifold $M$, there are two compression bodies $V$ and $W$ such that $M= V \cup_S W$ is obtained by gluing $V$ and $W$ on the closed surface $S$ where $S=\partial_+ V=\partial_+ W$.

The decomposition $M=V \cup_S W$ is called a {\it Heegaard splitting} and the orientable closed surface $S$ is called a {\it Heegaard surface} of $M$. The minimal genus of all Heegaard surfaces of $M$ is the {\it Heegaard genus} of $M$, denoted by $g(M)$. 
For an orientable compact 3-manifold $N$ with boundary, the minimal Heegaard genus of closed orientable $3$-dimensional manifolds into which $N$ is embeddable is denoted by $eg(N)$, called the {\it embeddable genus} of $N$.
It is shown in \cite[Proposition 3.1]{MO} that $eg(N)\le g(N)$, where $g(N)$ denotes the minimal genus of Heegaard splittings of $N$ in a sense of Casson--Gordon (\cite{CG}).

For a regular multibranched surface $X$, we define the {\it minimum genus} $\min g(X)$ and {\it maximum genus} $\max g(X)$ respectively as follows.

\[
\min g(X)=\min \{ eg(N) \mid N \in \mathcal{N}(X) \}
\]
\[
\max g(X)=\max \{ eg(N) \mid N \in \mathcal{N}(X) \}
\]

\subsection{Upper bound for the maximum and minimum genus}

The following theorems give a upper bound for the maximum and minimum genus of a multibranched surface.

\begin{theorem}[{\cite[Theorem 3.5]{MO}}]
For a regular multibranched surface $X$, we have the following inequality.
\[
\max g(X) \le |\mathcal{L}(X)| + |\mathcal{E}(X)|
\]
\end{theorem}

\begin{theorem}[{\cite[Theorem 3.6]{MO}}]
For a regular multibranched surface $X$, for any $N \in \mathcal{N}(X)$,
\[
eg(N) \le rank H_1(G_N) + g(\partial N),
\]
where $G_N$ denotes the abstract dual graph of $N$ (defined in \cite{MO}). Therefore, we have
\[
\max g(X) \le \max_{N \in \mathcal{N}(X)} \{rank H_1(G_N) + g(\partial N)\}
\]
\[
\min g(X) \le \min_{N \in \mathcal{N}(X)} \{rank H_1(G_N) + g(\partial N)\}
\]
\end{theorem}

\subsection{Lower bound for the maximum and minimum genus}

In the following theorem, we give a lower bound for the maximum and minimum genus of a multibranched surface.
For a union $S$ of closed orientable surfaces, we denote the sum of genus of each closed orientable surface by $g(S)$.

\begin{theorem}\label{lower}
For a regular multibranched surface $X$, we have the following inequalities.
\begin{eqnarray}
\min g(X) &\ge& rank H_1(X) - \max_{N \in \mathcal{N}(X)} g(\partial N)\\
\max g(X) &\ge& rank H_1(X) - \min_{N \in \mathcal{N}(X)} g(\partial N)
\end{eqnarray}
\end{theorem}

By Theorem \ref{lower}, we obtain the following inequality.

\begin{corollary}\label{3-sphere}
Let $X$ be a regular multibranched surface.
Suppose that $X$ can be embedded in $S^3$.
Then we have
\[
\max_{N \in \mathcal{N}(X)} g(\partial N) \ge rank H_1(X)
\]
\end{corollary}

Since both sides of the inequality in Corollary \ref{3-sphere} can be straightly calculated, this would be a criterion for a multibranched surface to be embedded in $S^3$.

\subsection{On the genera of multibranched surfaces of (graphs)$\times S^1$}

For a graph $G$, we obtain a regular multibranched surface by taking a product with $S^1$, that is, for each vertex $v_i$ of $G$, $v_i\times S^1$ forms a loop and for each edge $e_j$ of $G$, $e_j\times S^1$ forms an annulus.
We consider the genus of a regular multibranched surface which forms $G\times S^1$, and by using Theorem \ref{lower} show the following theorem which is an interplay of the genus of a graph $G$ and the genus of a multibranched surface $G\times S^1$.

The {\em minimum genus} $\min g(G)$ of a graph $G$ is defined as the minimal genus of closed orientable surfaces in which $G$ can be embedded.
We note that if a graph $G$ is embedded in a closed orientable surface $F$ with $g(F)=\min g(G)$, then $F-G$ consists of open disks.
The {\em maximum genus} $\max g(G)$ of a graph $G$ is defined as the maximal genus of closed orientable surfaces in which $G$ can be embedded and the complement of $G$ consists of open disks.
We remark that Xuong and Nebesk\'{y} determined the maximum genus of a graph by a completely combinatorial formula ({\cite[Theorem 3]{X}}, {\cite[Theorem 2]{N81}}).

At a glance, it seems to be difficult to determine the minimum and maximum genus of a given graph.
However, it can be combinatorially determined in principle.
Suppose that a graph $G$ is embedded in a closed orientable surface $F$ so that $F-G$ consists of open disks.
Then, the genus of a regular neighborhood $N(G;F)$ coincides with that of $F$, and 
 by the orientation of $F$, a rotation system of edges which are incident to $v$ is fixed for each vertex $v$ of $G$.
We remark that a rotation sysytem and a circular permutation system are identical concept.
Conversely, if a rotation sysytem is given for each vertex $v$ of $G$, then by assigning a disk and bands to the vertex $v$ and edges which are incident to $v$ respectively, we obtain an orientable disk-band surface $N$ up to homeomorphism.
Then, by capping $N$ off by open disks, we obtain a closed orientable surface $F$ in which $G$ is embedded.

By the above observation, we have the following.
\[
\min g(G)=\min \{ g(N) \mid N \in \mathcal{N}(G) \}
\]
\[
\max g(G)=\max \{ g(N) \mid N \in \mathcal{N}(G) \},
\]
where $\mathcal{N}(G)$ denotes the set of orientable disk-band surfaces for $G$.

\begin{theorem}\label{product}
For a graph $G$, we have the following equalities.
\begin{eqnarray}
\min g(G\times S^1) &=& 2 \min g(G)\\
\max g(G\times S^1) &=& 2 \max g(G)
\end{eqnarray}
\end{theorem}

\begin{remark}
In {\cite[Corollary 1.2]{T}}, it was shown that the minimal number $\dim H_1(M ;F)$ for a closed orientable 3-manifolds $M$ containing $G\times S^1$ equals to $2\min g(G)$, where $F=\Bbb{Z}_p$ or $\Bbb{Q}$.
It is well-known that $g(M)\ge \dim H_1(M ;F)$.
Hence the inequality $\min g(G\times S^1) \ge 2 \min g(G)$ of the first equality in Theorem \ref{product} holds.
\end{remark}

\subsection{Rank of the first homology group}

In this subsection, we assume that all sectors are orientable.
For a branch $l$ and a sector $s$ of a regular multibranched surface $X$, we define $d(l;s)=\sum_{c \subset \partial s} od(c)$, where $c$ is a prebranch attached to $l$.
The multibranched surface obtained by the removing a open disk from each sector is denoted by $\dot{X}$.

\begin{theorem}[{\cite[Theorem 4.1]{MO}}]\label{homology}
Let $X$ be a regular multibranched surface with $\mathcal{L}(X)=\left\{ l_1, \ldots, l_n \right\}$, $\mathcal{E}(X)=\left\{ s_1, \ldots, s_m \right\}$.
Then, 
\[
H_1(X) = \left[ l_1, \ldots, l_n : \sum_{k=1}^{n}d(l_k;s_1)l_k, \ldots, \sum_{k=1}^{n}d(l_k;s_m)l_k \right] \oplus \mathbb{Z}^{r'(X)} 
\] 
where $r'(X)=rank H_1(\dot{X})-n$.
\end{theorem}

In Theorem \ref{lower}, we need to calculate the rank of the first homology group of $X$.
By using Theorem \ref{homology}, we have 

\[
rank H_1(X) \ge rank H_1(\dot{X})-n.
\]

\subsection{Example for the equality}

In this subsection, we provide an example which satisfies the equality of Theorem \ref{lower}.

Let $\Gamma$ be a rose with $2n$ petals $(n\ge 1)$, where a {\em rose} is a topological space obtained by gluing $2n$ circles at a single point $p$. 
We consider the multibranched surface $X=\Gamma \times S^1$, where $X$ has single branch $l=p \times S^1$ and $2n$ sectors. 

Let $N_i\in \mathcal{N}(X)$ ($i=1, 2$) be a neighborhood of $X$ which is determined by {the} circular permutation of $l$, where $\Gamma_i \times S^1$ is a spine of $N(\Gamma \times S^1)$.
See Figure \ref{roses}.

\begin{figure}[htbp]\label{roses}
\begin{center}
\includegraphics[trim=0mm 0mm 0mm 0mm, width=.5\linewidth]{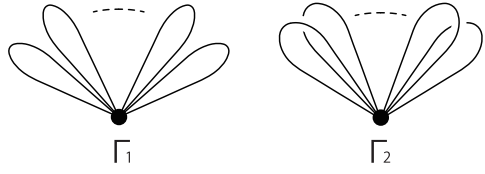}
\end{center}
\caption{Roses $\Gamma_1$ and $\Gamma_2$.}
\end{figure}

Then we have $rank H_1(X) = 2n+1$, $g(\partial N_1) = 2n+1$, $g(\partial N_2) = 1$.
Hence, 

\begin{eqnarray*}
\min g(X) &=& rank H_1(X) - g(\partial N_1)\\
 &=& (2n+1) - (2n+1)\\
 &=& 0\\
\end{eqnarray*}
and 
\begin{eqnarray*} 
\max g(X) &=& rank H_1(X) - g(\partial N_2)\\
&=& (2n+1) - 1\\
&=& 2n
\end{eqnarray*}

\subsection{Example for the inequality}

In this subsection, we provide an example which gives an arbitrary gap on the inequality of Theorem \ref{lower}.

Let $X_1$, $X_2$ be multibranched surfaces.
Take a disk $D_i$ in a sector of $X_i$ for $i=1,2$ and identify them (but do not remove the disk).
Then we have a new multibranched surface $X$ which is denoted by $X_1\#_D X_2$ and we call it a {\em disk sum} (\cite{GGH}).
It is clear that 
\[
rank H_1(X)=rank H_1(X_1) + rank H_1(X_2).
\]

If $N\in \mathcal{N}(X)$, then there exist $N_1\in \mathcal{N}(X_1)$ and $N_2\in \mathcal{N}(X_2)$ such that $N=N_1\#_{\partial} N_2$, where $\#_{\partial}$ means that we glue $N_1$ and $N_2$ along a disk on their boundaries.
It holds that 
\[
g(\partial N)=g(\partial N_1)+g(\partial N_2).
\]

If $N_1$ embed in $M_1$ and $N_2$ in $M_2$, then $N$ can be embedded in $M_1 \# M_2$. From this follows that
\[
\min g(X) \le \min g(X_1) + \min g(X_2)
\]
\[
\max g(X) \le \max g(X_1) + \max g(X_2)
\]

Now if we get an example satisfying $\displaystyle \min g(X) > rank H_1(X) - \max_{N \in \mathcal{N}(X)} g(\partial N)$, then by taking disk sums of $X$ with itself many times, we can get a gap in the inequality as large as we want.
For example, if $X$ is a spine of a lens space, this will do.
That is, if $X$ consists of a single curve as a branch, a unique sector which is a disk and which have an index $p$ around the branch.
In this case, the equality of the sum of Heegaard genus does happen since $\min g(X)=\max g(X)=1$, $rank H_1(X)=0$ and $\displaystyle \max_{N \in \mathcal{N}(X)} g(\partial N)=\min_{N \in \mathcal{N}(X)} g(\partial N)=0$.
Note that if $X_n$ is obtained by taking $n$ disk sums of $X$ with itself, then $g(X_n)=n$, because $X_n$ is the spine of the connected sun of $n$ lens spaces.

It is unknown whether the minimum genus and maximum genus of a multibranched surface are additive under disk sums.
This would follow the following question.

\begin{question}
If $N=N_1\#_{\partial} N_2$, where $N, N_1, N_2$ are compact orientable 3-manifolds with boundary, then $eg(N)=eg(N_1)+eg(N_2)$?
\end{question}

\section{Preliminaries}

\begin{lemma}\label{rank}
Let $M$ be a closed orientable 3-manifold and $F$ be a closed orientable surface which separates $M$ into two 3-submanifolds $M_1$ and $M_2$.
Then we have
\[
rank H_1(M)\ge rank H_1(M_1) +rank H_1(M_2) -rank H_1(F)
\]
\end{lemma}

\begin{proof}
We consider the following Mayer--Vietoris sequence.
\[
H_1(F) \xrightarrow{i} H_1(M_1) \oplus H_1(M_2) \xrightarrow{j} H_1(M)
\]
By the fundamental theorem on homomorphisms,
\[
H_1(M_1) \oplus H_1(M_2)/\ker j \cong \Ima j
\]
and we have
\[
rank H_1(M_1) \oplus H_1(M_2) - rank \ker j = rank \Ima j
\]
By the exactness of sequence, we have $\Ima i = \ker j $.
Hence we have the following by the above equations.
\begin{eqnarray*}
  rank H_1(M) & \ge &  rank \Ima j \\
   & = & rank H_1(M_1) \oplus H_1(M_2) - rank \ker j \\
   & = & rank H_1(M_1) + rank H_1(M_2) - rank \Ima i \\
   & \ge & rank H_1(M_1) + rank H_1(M_2) - rank H_1(F)
\end{eqnarray*}
\end{proof}

The following lemma is essentially same as {\cite[Theorem 1.3 (a)]{T}} with a slightly different settings.

\begin{lemma}[{\cite[Theorem 1.3 (a)]{T}}]\label{fundamental}
For a regular multibranched surface $X$ and a neighborhood $N \in \mathcal{N}(X)$, we have the following inequality.
\[
eg(N) \ge rank H_1(X) - g(\partial N)
\]
\end{lemma}

\begin{proof}
Suppose that $N$ is embedded in a closed orientable 3-manifold $M$ with $eg(N)=g(M)$.
Then $M$ is separated by $\partial N$ into two 3-submanifolds $N$ and say $Y$.
By Lemma \ref{rank}, we have
\[
rank H_1(M)\ge rank H_1(N) +rank H_1(Y) -rank H_1(\partial N)
\]
By an abelianization, we have
\[
g(M)\ge rank \pi_1(M) \ge rank H_1(M)
\]
On the other hand, since it holds generally that $rank H_1(Y) \ge \frac{1}{2} rank H_1(\partial N)$,
\begin{eqnarray*}
rank H_1(N) +rank H_1(Y) -rank H_1(\partial N) &\ge& rank H_1(N) - \frac{1}{2} rank H_1(\partial N)\\
&=& rank H_1(X) - g(\partial N)
\end{eqnarray*}
Hence we have the inequality of Lemma \ref{fundamental}.
\end{proof}

\begin{lemma}\label{Dehn}
Let $F$ be a closed orientable surface of genus $g$ and $p$ be a point in $F$.
Remove $int N(p\times S^1)$ from $F\times S^1$, and glue a solid torus $V$ with the remainder along their boundaries so that the curve $q\times S^1$ bounds a meridian disk of $V$, where $q$ is a point in $\partial N(p\times S^1)$.
Then the resultant 3-manifold $M$ has a Heegaard genus $2g$.
\end{lemma}

\begin{proof}
Let $a_1,b_1,\ldots, a_g,b_g$ be $2g$ arcs properly embedded in $F-int N(p)$ which cut $F-int N(p)$ into a single disk $D$.
To consider the resultant 3-manifold $M$, first we cut $F\times S^1-int N(p\times S^1)$ along $2g$ annuli $a_1\times S^1,b_1\times S^1,\ldots, a_g\times S^1,b_g\times S^1$.
Then we obtain a solid torus $D\times S^1$.
Since $\partial (a_1\times S^1),\partial (b_1\times S^1),\ldots, \partial (a_g\times S^1),\partial (b_g\times S^1)$ bound meridian disks in $V$, those annuli result 2-spheres, say $A_1,B_1,\ldots,A_g,B_g$, in the resultant 3-manifold $M$.
Since $a_1,b_1,\ldots, a_g,b_g$ are non-separating in $F-int N(p)$, $A_1,B_1,\ldots,A_g,B_g$ are also non-separating in $M$.
Let $M'$ be a 3-manifold obtained from $M$ by cutting along $A_1,B_1,\ldots,A_g,B_g$.
$M'$ can be obtained from $D\times S^1$ by gluing $2g$ 2-handles along $2g$ annuli  $\partial V-(a_1\times S^1\cup b_1\times S^1\cup \ldots \cup a_g\times S^1\cup b_g\times S^1)$.
We note that each of the $2g$ annuli goes around $\partial (D\times S^1)$ longitudinally once.
Hence $M'$ is a $2g$ punctured 3-sphere.
Now we glue along each copies of $A_1,B_1,\ldots,A_g,B_g$ and obtain $M=(S^2\times S^1)\#\cdots \#(S^2\times S^1)$ ($2g$ sums).
Thus $g(M)=g(S^2\times S^1)+\cdots + g(S^2\times S^1)=2g$.
\end{proof}

\section{Proofs}

\begin{proof}[Proof of Theorem \ref{lower}]
(1) Let $N \in \mathcal{N}(X)$ be a neighborhood of $X$ with $\min g(X) = eg(N)$.
By Lemma \ref{fundamental}, 
\begin{eqnarray*}
eg(N) &\ge& rank H_1(X) - g(\partial N)\\
&\ge& rank H_1(X) - \max_{N \in \mathcal{N}(X)} g(\partial N)
\end{eqnarray*}
Hence we have the inequality (1) of Theorem \ref{lower}.

(2) Let $N \in \mathcal{N}(X)$ be a neighborhood of $X$ with 
$\min_{N \in \mathcal{N}(X)} g(\partial N) = g(\partial N)$
By Lemma \ref{fundamental}, 
\begin{eqnarray*}
rank H_1(X) - g(\partial N) &\le& eg(N)\\
&\le& \max g(X) 
\end{eqnarray*}
Hence we have the inequality (2) of Theorem \ref{lower}.
\end{proof}

\begin{proof}[Proof of Theorem \ref{product}]
Let $G$ be a graph and put $X=G\times S^1$.
We remark that a neighborhood $N\in \mathcal{N}(X)$ does not depend on a slope system and depends only on a permutation system since $d(l)=1$ for each branch $l$.
In the following, we will show 
\[
eg(N(X;\mathcal{P})) = 2 g(G_{\mathcal{P}}),
\]
where $G_{\mathcal{P}}$ denotes a graph $G$ equiped with the same permutation system $\mathcal{P}$ of $X$.

($\le$)
Let $F$ be a closed orientable surface of genus $g(G_{\mathcal{P}})$.
Then $G_{\mathcal{P}}$ can be embedded in $F$ and $X_{\mathcal{P}}$ can be embedded in $F\times S^1$, where $X_{\mathcal{P}}$ denotes a multibranched surface $X$ equiped with the same permutation system $\mathcal{P}$ of $G_{\mathcal{P}}$.
We remark that $g(F\times S^1)=2 g(G_{\mathcal{P}}) +1$ (\cite{TO}).
Let $p$ be a point in $F-G_{\mathcal{P}}$.
Then by the Dehn surgery along $p\times S^1$ as in Lemma \ref{Dehn}, we obtain a closed orientable 3-manifold $M$ in which $X_{\mathcal{P}}$ is embedded.
By Lemma \ref{Dehn}, $g(M)=2g(G_{\mathcal{P}})$ and we have $eg(N(X;\mathcal{P})) \le 2 g(G_{\mathcal{P}})$.

($\ge$)
Let $M$ be a closed orientable 3-manifold of genus $eg(N(X;\mathcal{P}))$ in which $X_{\mathcal{P}}$ can be embedded.
We embedd $G_{\mathcal{P}}$ with the same permutation system $\mathcal{P}$ of $X_{\mathcal{P}}$ in a closed orientable surface $F$ so that $g(F)= g(G_{\mathcal{P}})$.
Then $F-G_{\mathcal{P}}$ consists of open disks.
By removing some edges of $G_{\mathcal{P}}$, there exsits a minor $G'_{\mathcal{P}}$ of $G_{\mathcal{P}}$ such that $F-G'_{\mathcal{P}}$ consists of a single open disk.
Let $X'_{\mathcal{P}}$ be the multibranched surface corresponding to $G'_{\mathcal{P}}\times S^1$.
By Lemma \ref{fundamental},
\[
eg(N(X'_{\mathcal{P}})) \ge rank H_1(X'_{\mathcal{P}}) - g(\partial N(X'_{\mathcal{P}}))
\]
Since $N(X'_{\mathcal{P}})$ is a product of a once punctured closed orientable surface of genus $g(F)$ with $S^1$, 
$rank H_1(X'_{\mathcal{P}}) = 2g(F)+1$ and $g(\partial N(X'_{\mathcal{P}}))=1$.
Hence we have
\[
eg(N(X_{\mathcal{P}})) \ge eg(N(X'_{\mathcal{P}})) \ge 2g(F) = 2 g(G_{\mathcal{P}})
\]

Finally by taking the minimum and maximum of the above equality, we have the equalities (3) and (4) of Theorem \ref{product}.
\end{proof}

\bigskip
\noindent{\bf Acknowledgements.}
The authors would like to thank to Arkadiy Skopenkov for informimg us of two related papers.

\bibliographystyle{amsplain}

\end{document}